\documentclass[12 pt]{amsart}
\usepackage{amsmath,mathrsfs,amsfonts,amssymb,xypic,fullpage,setspace, stmaryrd,hyperref, mathtools }
\usepackage{tikz}
\usepackage{wrapfig}
\usepackage{courier}

\usetikzlibrary{arrows,chains,matrix,positioning,scopes}
\makeatletter
\tikzset{join/.code=\tikzset{after node path={%
\ifx\tikzchainprevious\pgfutil@empty\else(\tikzchainprevious)%
edge[every join]#1(\tikzchaincurrent)\fi}}}
\makeatother
\tikzset{>=stealth',every on chain/.append style={join},
         every join/.style={->}}
\tikzstyle{labeled}=[execute at begin node=$\scriptstyle,
   execute at end node=$]

\usepackage[margin=.9in]{geometry}

\usepackage[all]{xy}

\DeclareFontFamily{U}{wncy}{}
\DeclareFontShape{U}{wncy}{m}{n}{<->wncyr10}{}
\DeclareSymbolFont{mcy}{U}{wncy}{m}{n}
\DeclareMathSymbol{\Sh}{\mathord}{mcy}{"58}

\newtheoremstyle{named}{}{}{\itshape}{}{\bfseries}{.}{.5em}{\thmnote{#3's }#1}
\theoremstyle{named}

\theoremstyle{plain}
\newtheorem{theorem}{Theorem}

\newtheorem{corollary}{Corollary}
\newtheorem{lemma}{Lemma}
\newtheorem{proposition}{Proposition}

\theoremstyle{remark}

\newtheorem{remark}{Remark}

\theoremstyle{definition}
\newtheorem{definition}{Definition}

\newcommand{\Q}{\mathbf{Q}}
\newcommand{\F}{\mathbf{F}}

\newcommand{\Z}{\mathbf{Z}}
\newcommand{\Zp}{\Z_p}

\newcommand{\Qp}{\mathbf{Q}_p}

\newcommand{\Selp}{\mathrm{Sel}_p}
\newcommand{\Tr}{\mathrm{Tr}}

\title{rank parity for congruent supersingular elliptic curves}

\date{\today}

\begin{document}

\begin{abstract}
A recent paper of Shekhar compares the ranks of elliptic curves $E_1$ and $E_2$ for which there is an isomorphism $E_1[p] \simeq E_2[p]$ as $\mathrm{Gal}(\bar{\Q}/\Q)$-modules, where $p$ is a prime of good ordinary reduction for both curves. In this paper we prove an analogous result in the case where $p$ is a prime of good supersingular reduction.
\end{abstract}

\author{Jeffrey Hatley}
\address[Jeffrey Hatley]{Department of Mathematics
Bailey Hall 202
Union College
Schenectady, NY 12308
Phone: (609) 238-5945
Fax: (518) 388-6005}
\email{hatleyj@union.edu}


\maketitle

\section{Introduction}

Let $p$ be a prime, write $G_\Q = \mathrm{Gal}(\bar{\Q}/\Q)$, and let $E_1$ and $E_2$ be two elliptic curves defined over $\Q$. We say that $E_1$ and $E_2$ are {\em congruent mod $p$} if there is an isomorphism of $G_\Q$-modules
\[
E_1[p] \simeq E_2[p].
\]
In this case, it is interesting to study what other properties of $E_1$ and $E_2$ are entangled. For instance, Greenberg and Vatsal \cite{GV} showed that if $p$ is a prime of good ordinary reduction for both curves, then information about the $p$-primary Selmer group (over the cyclotomic $\Z_p$-extension of $\Q$) of one curve often gives information about the $p$-primary Selmer group of the other. 

Let $N_i$ denote the conductor of $E_i$, let $\bar{N}_i$ denote the prime-to-$p$ Artin conductor of the Galois module $E_i[p]$ over $\Q$, and let $\Sigma$ be any finite set of primes containing $p$, $\infty$, and any prime at which $E_1$ or $E_2$ has bad reduction. Inside of $\Sigma$ we find a subset 
\begin{equation}\label{sigma0}
\Sigma_0 = \{ v \in \Sigma : v \mid N_1 / \bar{N}_1 \ \mathrm{or} \ v \mid N_2 / \bar{N}_2\};
\end{equation}
 The strategy of \cite{GV} was used in a recent paper of Shekhar \cite{Shek} to prove that if $E_1[p] \simeq E_2[p]$ for $p$ a prime of good ordinary reduction, then the analytic rank of $E_1$ over $\Q$ is related to the analytic rank of $E_2$ over $\Q$. To be more precise, write $r^{an}(E_i / \Q)$ for the analytic rank of $E_i$ over $\Q$, and let $S_i$ be the set of primes in $\Sigma_0$ for which $E_i$ has split multiplicative reduction. Let $|S_i|$ denote the cardinality of $S_i$. Then under some additional technical hypotheses, Shekhar proved that 
\[
E_1[p] \simeq E_2[p] \Longrightarrow r^{an}(E_1 / \Q) + |S_1| \equiv r^{an}(E_2 / \Q) + |S_2| \mod 2.
\]

The proof of this result relies on Iwasawa-theoretic techniques. It has long been known that Iwasawa theory for elliptic curves has fewer complications when working with primes of good {\em ordinary} reduction; however, every elliptic curve over $\Q$ also has infinitely many good supersingular primes \cite{Elk}, so these should not be neglected. In this paper, we prove an analogue of Shekhar's result in the supersingular setting. 

\begin{theorem}
Let $E_1$ and $E_2$ be elliptic curves over $\Q$, and suppose $p \geq 3$ is a prime such that
\begin{itemize}
\item $E_1[p]$ and $E_2[p]$ are isomorphic as $G_\Q$-modules, and
\item $E_1$ (and hence also $E_2$) is supersingular at $p$.
\end{itemize}
Then there exist two explicit and computable finite sets of primes $S_1$ and $S_2$ such that 
\[
r^{an}(E_1 / \Q) + |S_1| \equiv r^{an}(E_2 / \Q) + |S_2| \mod 2.
\]

\end{theorem}

See Theorem \ref{main-theorem} for a more precise statement. 

\subsection{Outline} We follow the strategy of \cite{Shek} closely, but we must overcome some complications that arise from working with supersingular primes. We begin by reviewing the theory of $\pm$-Selmer groups, recalling some important results on algebraic Iwasawa invariants in the supersingular setting. After translating many results into the setting of supersingular elliptic curves, we obtain our main theorem as a consequence of the parity conjecture. We conclude with some explicit examples.

\subsection{Notation}

Throughout this paper we fix an odd prime $p \geq 3$. Write $\Q_\infty$ for the cyclotomic $\Z_p$-extension of $\Q$, and write $\Q_n$ for the $n$th layer of this extension, so the Galois group $G_n := \mathrm{Gal}(\Q_n / \Q) \simeq \Z / p^n \Z$ and $G_\infty:= \mathrm{Gal}(\Q_\infty / \Q) \simeq \Z_p$. We denote the Iwasawa algebra by $\Lambda = \Zp[[G_\infty]]$. 

\subsection{Assumptions}

Throughout the paper, $E$ will refer to an elliptic curve such that 
\begin{itemize}
\item $E$ is defined over $\Q$, and
\item $E$ is supersingular at $p$.
\end{itemize} 
$E_1$ and $E_2$ will refer to a pair of such curves which are also congruent mod $p$. The assumption that $a_p(E)=0$ is necessary to apply some crucial results of Kim \cite{Kim09}. While many of the preliminary results can be proven in slightly more generality, we will need the full strength of Kim's results to prove our main theorem, so we set this assumption at the beginning. It is possible that this hypothesis can be removed by using ideas of Sprung \cite{Spr}, but the statement of the theorem would likely be more complicated, and since $a_p(E)=0$ is automatic whenever $p \geq 5$, we work under this slightly more restrictive setting.

\section{Algebraic Iwasawa Invariants}

When studying the Iwasawa theory of an elliptic curve at a prime $p$, a careful distinction must be made depending on whether $p$ is {\em ordinary} or {\em supersingular}. Recall that $E$ is supersingular at $p$ if $p \mid a_p(E)$; by the Riemann hypothesis for elliptic curves over finite fields, if $p \geq 5$ then this is equivalent to $a_p(E) = 0$. We begin with an overview of the ordinary case before initiating our study of the supersingular setting.
 
\subsection{The ordinary case}\label{ordcase}

We quickly recall some definitions and facts about Selmer groups of elliptic curves in the ordinary case; since this is all standard, we omit references and direct the interested reader to \cite{Gr} or \cite{Gr2} for more details. Let $E$ be an elliptic curve defined over $\Q$ with good ordinary reduction at $p$, and let $\Sigma$ be any finite set of primes containing $p$, $\infty$, and all primes of bad reduction for $E$. For every $n \geq 0$, the $p$-Selmer group of $E$ over $\Q_n$, denoted Sel$_p(E/\Q_n)$, is defined as
\begin{align*}
\mathrm{Sel}_p(E / \Q_n) &= \ker \left( H^1(\Q_n, E[p^\infty]) \longrightarrow \displaystyle\prod_{\ell} \mathcal{H}_\ell(\Q_n) \right) \\ 
&= \ker \left( H^1(\Q_\Sigma / \Q_n, E[p^\infty]) \longrightarrow \displaystyle\prod_{\ell \in \Sigma} \mathcal{H}_\ell(\Q_n) \right)
\end{align*}
where 
\begin{equation}\label{hldefn}
\mathcal{H}_\ell (\Q_n):= 
\displaystyle\prod_{\eta \mid \ell} \frac{H^1(\Q_{n,\eta},E[p^\infty])}{E(\Q_{n,\eta}) \otimes \Q_p / \Z_p}.
\end{equation}
Here $\Q_{n,\ell}$ is the completion of $\Q_n$ at $\ell$; note that $E(\Q_{n,\ell}) \otimes \Q_p / \Z_p = 0$ for $\ell \neq p$ \cite[Proposition 2.1]{Gr}. Similarly we define
\begin{align*}
\Selp(E/\Q_\infty) &= \varinjlim \Selp(E / \Q_n) \\ 
&= \ker \left( H^1(\Q_\Sigma / \Q_\infty, E[p^\infty]) \longrightarrow \displaystyle\prod_{\ell \in \Sigma} \mathcal{H}_\ell(\Q_\infty) \right),
\end{align*}
where
\[
\mathcal{H}_\ell(\Q_\infty) := \displaystyle\prod_{\eta \mid \ell} \frac{H^1(\Q_{\infty,\eta},E[p^\infty])}{E(\Q_{\infty,\eta}) \otimes \Q_p / \Z_p}.
\]

\begin{remark}\label{divisible} All finite primes are finitely decomposed in $\Q_\infty$, so this is a finite product. Since $E[p^\infty]$ is a divisible group and $G_{\Q_{\infty,\eta}}$ has $p$-cohomological dimension $1$, it follows from \cite[Lemma 4.5]{Gr} that $\mathcal{H}_\ell(\Q_\infty)$ is a divisible group.
\end{remark}

The Pontryagin dual of $\Selp(E/\Q_\infty)$, denoted $\mathcal{X}(E / \Q_\infty$), is a finitely-generated $\Lambda$-module. 

Suppose for this subsection that $p$ is a prime of good ordinary reduction for $E$. In this case, Kato proved that $\mathcal{X}(E / \Q_\infty)$ is actually a {\em torsion} $\Lambda$-module (and we say that $\Selp(E/\Q_\infty)$ is $\Lambda$-cotorsion). Recall that $\Lambda$ can be identified with a formal power series ring $\Zp[[T]]$ under the map $\gamma \mapsto (T+1)$ for $\gamma$ a topological generator of $G_\infty$. The structure theorem for $\Lambda$-modules then implies that there is a map
\[
\mathcal{X}(E/\Q_\infty) \rightarrow \left( \bigoplus_{i=1}^n \Lambda / (f_i(T))^{a_i} \right) \oplus \left( \bigoplus_{j=1}^m \Lambda / (p^{\mu_j}) \right)
\]
with finite kernel and cokernel. (Such a map is called a pseudo-isomorphism; if $X$ is pseudo-isomorphic to $Y$, then we write $X \sim Y$.) The $a_i$ are positive integers are the $f_i$ are irreducible monic polynomials

 The {\em algebraic Iwasawa invariants} of $E$ are defined as
\begin{equation}\label{invars}
\lambda_E = \sum_{i=1}^n a_i \cdot \mathrm{deg}(f_i(T)) \quad \mathrm{and} \quad \mu_E = \sum_{j=0}^m \mu_j.
\end{equation}
The {\em characteristic polynomial} for $\mathcal{X}(E/\Q_\infty)$ is then
\[
f_E(T) = p^{\mu_E} \cdot \displaystyle\prod_{i=1}^{n} f_i(T)^{a_i}.
\]

It is important to note that the $p$-Selmer group of an elliptic curve is determined by $E[p^\infty]$ but not by $E[p]$. One of the key observations of \cite{GV} is that it is useful to define a ``non-primitive'' Selmer group
\[
\Selp^{\Sigma_0} (E / \Q_\infty) = \ker \left( H^1(\Q_\Sigma / \Q_\infty, E[p^\infty]) \longrightarrow \displaystyle\prod_{\ell \in \Sigma \setminus \Sigma_0} \mathcal{H}_\ell(\Q_\infty) \right),
\]
where $\Sigma_0$ is a subset of $\Sigma$ not containing $p$ or $\infty$.

Greenberg conjectures that $\mu_E=0$ whenever $p$ is ordinary and $E[p]$ is absolutely irreducible. If $\mu_E$ vanishes, there are several important consequences. The first is that both $\Selp(E / \Q_\infty)[p]$ and $\Selp^{\Sigma_0}(E / \Q_\infty)[p]$ are finite. One can then show that both 
$\Selp(E / \Q_\infty)$ and $\Selp^{\Sigma_0}(E / \Q_\infty)$ are divisible groups, with  
$\Selp(E /\Q_\infty) \simeq (\Q_p / \Z_p)^{\lambda_E}$. Finally, the $\F_p$-dimension of $\Selp(E / \Q_\infty)[p]$ is exactly $\lambda_E$, and we denote the $\F_p$-dimension of $\Selp^{\Sigma_0}(E / \Q_\infty)[p]$ by $\lambda_{E}^{\Sigma_0}$. 

Furthermore, one has a split exact sequence
\begin{equation*}\label{primitiveexact}
0 \rightarrow \Selp(E / \Q_\infty) \rightarrow \Selp^{\Sigma_0}(E / \Q_\infty) \rightarrow \displaystyle\prod_{\ell \in \Sigma_0} \mathcal{H}_\ell(\Q_\infty, E[p^\infty]) \rightarrow 0,
\end{equation*}
so that 
\begin{equation}\label{seldifford}
\Selp^{\Sigma_0}(E/ \Q_\infty) / \Selp(E/ \Q_\infty) \simeq \displaystyle\prod_{\ell \in \Sigma_0} \mathcal{H}_\ell(\Q_\infty) \simeq (\Q_p / \Z_p)^{\delta(E,\Sigma_0)}
\end{equation}
where $\delta(E,\Sigma_0)$ is a non-negative integer. Thus 
\[
\lambda_{E}^{\Sigma_0} = \lambda_E + \delta(E,\Sigma_0).
\] 
One of the main results of \cite{GV} is that if $E_1[p] \simeq E_2[p]$ and $\Sigma_0$ contains all primes of bad reduction, then
\[
\Selp^{\Sigma_0} (E / \Q_\infty) \simeq \Selp^{\Sigma_0} (E_2 / \Q_\infty),
\]
and so
\begin{equation}\label{lambdasord}
\lambda_{E_1} + \delta(E_1, \Sigma_0) = \lambda_{E_2} + \delta(E_2, \Sigma_0).
\end{equation}

Thus, knowledge of the Iwasawa invariants for one elliptic curve reveals information about the Iwasawa invariants of congruent elliptic curves. 

A result of Guo (see \cite[Proposition 3.10]{Gr}) shows that 
\begin{equation}\label{corankparity}
\mathrm{corank}_{\Z_p} \ \Selp(E/ \Q) \equiv \lambda_E \mod 2.
\end{equation}
The main idea in \cite{Shek} is to use (\ref{lambdasord}) and (\ref{corankparity}) in conjunction with the parity conjecture to obtain a result about the variation of rank parity among congruent elliptic curves.

\subsection{The supersingular case}\label{sscase} For the rest of the paper, we now suppose that $p$ is a prime of good {\em supersingular} reduction for an elliptic curve $E$ defined over $\Q$. 

In the supersingular setting, we can and do make all of the same definitions as in Section \ref{ordcase}, but it turns out that $\Selp(E / \Q_\infty)$ is {\em never} $\Lambda$-cotorsion. To obtain results analogous to those in the ordinary setting, we must work with {\em plus/minus Selmer groups}. These refined Selmer groups were first studied by Kurihara \cite{Kur} and Kobayashi \cite{Kob}.

 Retain the definition of $\Sigma$ from Section \ref{ordcase}. For any $n \geq 0$, we write $\Q_{n,p}$ for the completion of $\Q_n$ at the unique prime of $\Q_\infty$ above $p$. For notational convenience, also let $\Q_{-1,p} = \Q_{0,p} = \Q_p$. For $n \geq m$, let $\Tr_{n,m}$ denote the trace map from $E(\Q_{n,p})$ to $E(\Q_{m,p})$. Define the finite-level {\em plus/minus norm groups} as follows:
\begin{align*}
E^+(\Q_{n,p}) &= \{ x \in E(\Q_{n,p}) \mid \Tr_{n,m+1} (x) \in E(\Q_{m,p}) \textrm{ for every } 0 \leq m \leq n, \ m \ \mathrm{even} \}, \\
E^-(\Q_{n,p}) &= \{ x \in E(\Q_{n,p}) \mid \Tr_{n,m+1} (x) \in E(\Q_{m,p}) \textrm{ for every } -1 \leq m \leq n, \ m \ \mathrm{odd} \}.
\end{align*}
Furthermore, we define the infinite-level plus/minus norm groups:
\[
E^\pm(\Q_{\infty,p}) = \displaystyle\cup_{n \geq 0} E^\pm(\Q_{n,p}).
\]
For every integer $n \geq 0$, set
\[
\mathbf{H}_{n,p}^\pm = (E^\pm(\Q_{\infty,p})\otimes \Q_p / \Z_p)^{\mathrm{Gal}(\Q_\infty / \Q_n)} \subset H^1(\Q_{n,p}, E[p^\infty]).
\]
Define the plus/minus cohomology subgroups $\mathcal{H}^\pm_\ell$ for each prime $\ell$ by
\[
\mathcal{H}_\ell^{\pm}(\Q_n) = 
\begin{cases} 
       \mathcal{H}_\ell(\Q_n) & \ell \neq p \\
       & \\ 
      \displaystyle\frac{H^1(\Q_{n,p},E[p^\infty])}{\mathbf{H}_{n,p}^\pm} & \ell = p \\
\end{cases}
\]
Then we define the plus/minus Selmer groups in a manner completely analogous to the ordinary case.

\begin{definition} For $n \geq 0$ or $n=\infty$, we define
\[
\Selp^\pm (E / \Q_n) = \ker \left( H^1(\Q_\Sigma / \Q_n, E[p^\infty]) \to \displaystyle\prod_{\ell \in \Sigma} \mathcal{H}^\pm_\ell(\Q_n)  \right).
\]
\end{definition}

\begin{remark}\label{selfdualannihilator} Recall that if $T$ is the $p$-adic Tate module of $E$, and if we set and $V = T \otimes_{\Zp} \Qp$ and $A = V / T$, then $A \simeq E[p^\infty]$. By a result of Kim \cite[Proposition 3.4]{Kim13}, the local conditions $\mathbf{H}_{n,p}^\pm$ are their own annihilators under the autoduality of $E$ and the Tate pairing
\[
H^1(\Q_{n,p},A) \times H^1(\Q_{n,p},T(1)) \rightarrow \Q_p / \Z_p.
\]
This fact plays a crucial role in the proof of Proposition 6, since it allows us to invoke a theorem of Flach which guarantees the existence of a non-degenerate, skew-symmetric pairing on our Selmer groups.
\end{remark}

\begin{remark}
The Selmer groups we have defined agree with Kobayashi's at infinite level, but at finite levels they may differ. A discussion of this can be found in \cite[Remark A.1]{PolWes}. In this paper, we only use a result of Kobayashi at infinite level, and all finite-level results rely on the work of Kim, whose conventions we have adopted.
\end{remark}

\begin{remark}\label{Sel0}
Note that for $n=0$ we have
\[
E^-(\Q_{n,p}) = E(\Q_p),
\]
hence $\Selp^\pm(E / \Q) = \Selp(E / \Q)$.
\end{remark}

These refined Selmer groups $\Selp^\pm(E / \Q_n)$ have many of the desirable properties that $\Selp(E / \Q_n)$ has in the ordinary setting. In particular, Kobayashi proved

\begin{theorem}[\cite{Kob}, Theorem 2.2]
The Pontryagin dual $\mathcal{X}^\pm(E / \Q_\infty)$ of $\Selp^\pm(E / \Q_\infty)$ is a finitely-generated torsion $\Lambda$-module.
\end{theorem}
We therefore have pseudo-isomorphisms
\[
\mathcal{X}^+(E/\Q_\infty) \sim \left( \bigoplus \Lambda / (f_i(T))^{a_i} \right) \oplus \left( \bigoplus \Lambda / (p^{\mu_j}) \right)
\]
and
\[
\mathcal{X}^-(E/\Q_\infty) \sim \left( \bigoplus \Lambda / (g_k(T))^{b_k} \right) \oplus \left( \bigoplus \Lambda / (p^{\mu_l}) \right).
\]
In analogy with (\ref{invars}), we can define pairs of Iwasawa invariants.

\begin{definition}
\begin{align*}
\lambda_E^{+} &= \sum a_i \cdot \mathrm{deg}(f_i(T)), &\mu_E^{+} &= \sum \mu_j,\\
\lambda_E^{-} &= \sum b_k \cdot \mathrm{deg}(g_k(T)), &\mu_E^{-} &= \sum \mu_l.
\end{align*}
\end{definition}

\noindent \textbf{Assumption:} We now assume for the rest of the paper that $\mu_E^{\pm} = 0$.

\begin{remark} The $\mu$-invariants are always expected to vanish \cite[Conjecture 7.1]{P-R}, so this assumption is very mild.
\end{remark}

In \cite{Kim09}, Kim shows that many of the results from the ordinary case carry over to the supersingular setting provided one uses the plus/minus Selmer groups. If $\mu_E^{\pm}=0$, then just as in the ordinary case we have
\[
\lambda_E^{\pm} = \mathrm{dim}_{\F_p} \ \Selp^\pm(E / \Q_\infty)[p] = \mathrm{corank}_{\Z_p} \ \Selp^\pm(E / \Q_\infty).
\]

We can also define non-primitive plus/minus Selmer groups. As in Section \ref{ordcase}, let $\Sigma_0$ be a subset of $\Sigma$ which does not contain $p$ or $\infty$.

\begin{definition} For $n \geq 0$ or $n = \infty$, we define
\[
\Selp^{\Sigma_0,\pm} (E / \Q_n) = \ker \left( H^1(\Q_\Sigma / \Q_n, E[p^\infty]) \to \displaystyle\prod_{\ell \in \Sigma \setminus \Sigma_0} \mathcal{H}^\pm_\ell(\Q_n)  \right).
\]
\end{definition}
\noindent In analogy with (\ref{seldifford}), we have
\begin{proposition}[\cite{Kim09}, Corollary 2.5]
\begin{equation}\label{seldiffss}
\Selp^{\pm,\Sigma_0}(E/ \Q_\infty) / \Selp^\pm(E/ \Q_\infty) \simeq \displaystyle\prod_{\ell \in \Sigma_0} \mathcal{H}^\pm_\ell(\Q_\infty). 
\end{equation}
\end{proposition}

\begin{remark}
We point out that this quotient is independent of the choice of $\pm$ since $p \notin \Sigma_0$, so
\[
\Selp^{\pm,\Sigma_0}(E/ \Q_\infty) / \Selp^\pm(E/ \Q_\infty) \simeq  \displaystyle\prod_{\ell \in \Sigma_0} \mathcal{H}_\ell(\Q_\infty)\simeq (\Q_p / \Z_p)^{\delta(E,\Sigma_0)}
\]
just as in the ordinary case.
\end{remark}

We define 
\[
\lambda_{E}^{\pm,\Sigma_0} := \mathrm{dim}_{\F_p}\ \Selp^{\pm,\Sigma_0}(E/ \Q_\infty)[p].
\]
Kim has shown that the strategy of Greenberg and Vatsal carries over to the supersingular setting.
\begin{proposition}[\cite{Kim09}, Corollary 2.13]\label{compareprim} Suppose $p$ is a prime of good supersingular reduction for $E_1$ and $E_2$ such that $E_1[p] \simeq E_2[p]$ as $G_\Q$-modules. Then $\mu_{E_1}^{\pm}=0$ if and only if $\mu_{E_2}^{\pm}=0$. If both $\mu$-invariants vanish, then $\lambda_{E_1}^{\pm,\Sigma_0} = \lambda_{E_2}^{\pm,\Sigma_0}$.
\end{proposition}

We thus obtain an analogue of (\ref{lambdasord}) in the supersingular setting.

\begin{proposition}\label{mainprop}
Suppose $p$ is a prime of good supersingular reduction for $E_1$ and $E_2$ such that $E_1[p] \simeq E_2[p]$ as $G_\Q$-modules. If $\mu_{E_1}^{\pm}=0$ (or equivalently $\mu_{E_2}^{\pm}=0$), then
\begin{equation}\label{deltarelation}
\lambda_{E_1}^{\pm} + \delta(E_1,\Sigma_0)= \lambda_{E_2}^{\pm} + \delta(E_2,\Sigma_0)
\end{equation}
where $\delta(E_i,\Sigma_0)$ is a non-negative integer.
\end{proposition}
\begin{proof}
This follows immediately from (\ref{seldiffss}) and Proposition \ref{compareprim}.
\end{proof}

We will study the integers $\delta(E_i,\Sigma_0)$ in more detail in the next section.

\section{Local Galois cohomology: $\delta(E,\ell)$}

In this section we study the local Galois cohomology groups
\[
\mathcal{H}_\ell(\Q_\infty) \simeq (\Q_p / \Z_p)^{\delta(E,\ell)} 
\]
for primes $\ell \neq p$ (see Remark \ref{divisible}). Once we have determined these values, we can compute the values
\[
\delta(E,\Sigma_0) = \sum_{\ell \in \Sigma_0} \delta(E,\ell)
\]
from equation (\ref{deltarelation}). In fact, we will ultimately be interested in this value mod $2$.

For each prime $\eta \mid \ell$ of $\Q_\infty$, let $\tau(E,\eta)$ be defined by
\[
H^1(\Q_{\infty,\eta},E[p^\infty]) \simeq (\Q_p / \Z_p)^{\tau(E,\eta)},
\]
so that by (\ref{hldefn}) we have
\[
\delta(E,\ell) = \displaystyle\sum_{\eta \mid \ell} \tau(E,\eta).
\]
The following lemma is proved in a discussion in \cite[Section 3]{Shek}.
\begin{lemma}\label{taulemma}
For any $\eta \mid \ell$ of $\Q_\infty$,
\[
\delta(E,\ell) \equiv \tau(E,\eta) \mod 2.
\]
\end{lemma}
\begin{proof}
Let $\eta_1, \eta_2$ be two primes of $\Q_\infty$ over $\ell$. Any map $\phi \in G_\infty = \mathrm{Gal}(\Q_\infty / \Q)$ which satisfies $\phi(\eta_1) = \eta_2$ induces an isomorphism $H^1(\Q_{\infty,\eta_1},E[p^\infty]) \simeq H^1(\Q_{\infty,\eta_2},E[p^\infty])$, so $\tau(E,\eta_1) = \tau(E,\eta_2)$. Let $s_\ell$ denote the number of primes of $\Q_\infty$ over $\ell$. Then this implies $\delta(E,\ell) = s_l \cdot \tau(E,\eta)$ for any $\eta \mid \ell$. But $s_\ell$ is equal to the index of the {\em open} decomposition group $D_\ell$ of $\ell$ in $\Q_\infty$, hence $s_\ell$ is a power of $p$. In particular, $s_\ell$ is odd, which proves the lemma.
\end{proof}

To compute $\tau(E,\ell)$, we first need to introduce some notation. Let $T_p$ denote the $p$-adic Tate module of $E$, and let $V_p = T_p \otimes \Q_p$. Write $\Q_\ell^\mathrm{unr}$ for the maximal unramified extension of $\Q_\ell$. Let $I_\ell:= \mathrm{Gal}(\bar{\Q}_\ell / \Q_\ell^\mathrm{unr})$ denote the inertia subgroup of $\mathrm{Gal}(\bar{\Q}_\ell / \Q_\ell)$,  and let $\mathrm{Frob}_\ell$ denote the arithmetic Frobenius automorphism of $\mathrm{Gal}(\Q_\ell^\mathrm{unr} / \Q_\ell)$. Finally, let $(V_p)_{I_\ell}$ denote the maximal quotient of $V_p$ on which $I_\ell$ acts trivially. The following proposition of Greenberg-Vatsal explains how to compute $\tau(E,\eta)$. 

\begin{proposition}\label{taucomputation}
Let $P_\ell(X) = \mathrm{det}((1-\mathrm{Frob}_\ell X \vert_{(V_p)_{I_\ell}})) \in \Z_p [X]$. Let $\ \tilde{}$ denote reduction mod $p$. Then $\tau(E,\eta)$ is equal to the multiplicity of $X=\tilde{\ell}^{-1}$ as a root of $\tilde{P}_\ell \in \F_p[X]$.
\end{proposition}

\begin{proof}
This is \cite[Proposition 2.4]{GV}.
\end{proof}

We can now compute the parity of $\tau(E,\ell)$ for all $\ell \in \Sigma_0$. First we deal with primes of good reduction.

\begin{lemma}\label{taugood}
If $\ell$ is a prime of good reduction, then $\tau(E,\eta) \equiv 1 \mod 2$ if and only if $\ell+ 1 \equiv a_\ell(E) \mod p$ and $\ell^{-1} \not\equiv 1 \mod p$.
\end{lemma}
\begin{proof}
If $\ell$ is a prime of good reduction, then $(V_p)_{I_\ell}=V_p$ is two-dimensional, and it is a standard result that $P_\ell$ is given by
\begin{equation*}\label{frobcharbad}
P_\ell(X) = 1 - a_\ell(E)X + \ell X^2.
\end{equation*}
By Proposition \ref{taucomputation}, $\tau(E,\eta)$ is odd if and only if $\tilde{\ell}^{-1}$ is a simple root of $\tilde{P}_\ell(X)$. 

Since the product of the roots is $\ell^{-1}$, we have
\[
\tilde{\ell}^{-1} \ \text{is a root} \Longleftrightarrow 1 \ \text{is a root},
\]
and since the sum of the roots is $\ell^{-1} a_\ell(E)$, this implies
\[
\tilde{\ell}^{-1} \ \text{is a root} \Longleftrightarrow \ell + 1 \equiv a_\ell(E) \mod p.
\]
In this case, we see that $\tilde{\ell}^{-1}$ is a {\em simple} root if and only if $\ell^{-1} \not\equiv 1 \mod p$, which proves the lemma.

\end{proof}

We now handle primes of bad reduction.

\begin{lemma}\label{taulemma2}
$\tau(E,\eta) \equiv 1 \mod 2$ if and only if one of the following is true:
\begin{enumerate}
\item $\ell$ is a prime of split multiplicative reduction and $\ell^{-1} \equiv 1 \mod p$ 
\item $\ell$ is a prime of nonsplit multiplicative reduction and $\ell^{-1} \equiv -1 \mod p$ 
\end{enumerate}
\end{lemma}

\begin{proof}
If $\ell$ is a prime of additive reduction, then $(V_p)_{I_\ell}=0$, so Proposition \ref{taucomputation} shows that $\tau(E,\eta) = 0$.

Now suppose $\ell$ is a prime of multiplicative reduction. In this case, $(V_p)_{I_\ell}$ is one-dimensional, and the trace of $\mathrm{Frob}_{\ell}$ on $(V_p)_{I_\ell}$ is still given by $a_\ell(E)$, so that $P_\ell(X) = 1 - a_\ell(E)X$. We have
\[
a_l(E) = \begin{cases} 
       1 & \mathrm{if}\ \ell \ \mathrm{is}\ \mathrm{split,} \\
      -1 & \mathrm{if}\ \ell \ \mathrm{is}\ \mathrm{nonsplit.} \\
\end{cases}
\]
It follows that, in the notation of Proposition \ref{taucomputation},
\[
P_\ell(\ell^{-1}) = 1 \mp \ell^{-1}
\]
with $\pm$ determined by whether $\ell$ is split or nonsplit. The lemma now follows from Proposition \ref{taucomputation}. 
\end{proof}

We summarize this section of the paper.

\begin{proposition}\label{deltacomp}
For $\ell \in \Sigma_0$, we have $\delta(E,\ell) \equiv 1 \mod 2$ if and only if one of the following is true:
\begin{enumerate}
\item $\ell$ is a prime of good reduction such that $\ell+ 1 \equiv a_\ell(E) \mod p$ and $\ell^{-1} \not\equiv 1 \mod p$
\item $\ell$ is a prime of split multiplicative reduction and $\ell^{-1} \equiv 1 \mod p$ 
\item $\ell$ is a prime of nonsplit multiplicative reduction and $\ell^{-1} \equiv -1 \mod p$ 
\end{enumerate}
\end{proposition}
\begin{proof}
This follows immediately from Lemmas \ref{taulemma} and \ref{taulemma2}.
\end{proof}

Given two curves $E_1, E_2$ which are congruent mod $p$, let $\Sigma_0$ be a set of primes as in (\ref{sigma0}).

\begin{definition}\label{sidef}
Define $S_i$ to be the subset of $\ell \in \Sigma_0$ such that one of the following is true:
\begin{enumerate}
\item $\ell$ is a prime of good reduction for $E_i$ such that $\ell+ 1 \equiv a_\ell(E) \mod p$ and $\ell^{-1} \not\equiv 1 \mod p$
\item $\ell$ is a prime of split multiplicative reduction for $E_i$ and $\ell^{-1} \equiv 1 \mod p$ 
\item $\ell$ is a prime of nonsplit multiplicative reduction for $E_i$ and $\ell^{-1} \equiv -1 \mod p$ 
\end{enumerate}
\end{definition}

Thus, $S_i$ is the set of $\ell \in \Sigma_0$ for which $\delta(E_i,\ell) \equiv 1 \mod p$.

\section{Main Result}\label{section-main}

In order to prove our main result, we will need the following analogue of \cite[Proposition 3.10]{Gr}. We follow Greenberg's argument closely. 

\begin{proposition}\label{prop:lambda-parity}
Let $E$ be an elliptic curve, defined over $\Q$ and supersingular at $p$. 

Then 
\[
\mathrm{corank}_{\Z_p}\ \Selp(E / \Q) \equiv \lambda_E^\pm \ \mod 2.
\]
\end{proposition}

\begin{proof}

By Kobayashi's control theorem \cite[Theorem 9.3]{Kob}, for $n \geq 0$ we have that $\mathrm{corank}_{\Z_p} \ \Selp^\pm (E/\Q_n)$ is bounded above by $\lambda_E^\pm$. Let $\bar{\lambda}^\pm$ denote the maximum of these $\Z_p$-coranks, so that $\mathrm{corank}_{\Z_p}\ \Selp^\pm(E/\Q_n) = \bar{\lambda}^\pm$ for all $n \gg 0$. Let $S_n^\pm=\Selp^\pm(E / \Q_n)$, let $T_n^\pm = (S_n^\pm)_{\mathrm{div}}$, and let $U_n^\pm=S_n^\pm/T_n^\pm$. The restriction maps
\[
S_0^\pm \to (S_n^\pm)^{G_n} \quad \mathrm{and} \quad T_0^\pm \to (T_n^\pm)^{G_n}
\]
have finite kernels and cokernels. Since the nontrivial $\Q_p$-irreducible representations of $G_n$ have degree divisible by $p-1$, we have
\[
\mathrm{corank}_{\Z_p}\ T_n^\pm \equiv \mathrm{corank}_{\Z_p}\ T_0^\pm \mod (p-1),
\]
 and since $p$ is odd this gives a congruence mod $2$. Since $S_0^\pm = \Selp (E/ \Q)$ (see Remark \ref{Sel0}), it follows that
\begin{equation}\label{lambdapar}
\mathrm{corank}_{\Z_p}\ \Selp(E / \Q) \equiv \bar{\lambda}^\pm \ \mod 2.
\end{equation}

Let $S_\infty^\pm = \Selp^\pm(E / \Q_\infty)$ and $T_\infty^\pm = \varinjlim T_n$. By the definition of $\bar{\lambda}^\pm$ we have $T^\pm_\infty \simeq (\Q_p / \Z_p)^{\bar{\lambda}^\pm}$. Write $U^\pm_\infty = \varinjlim U_n^\pm = S_n^\pm / T_n^\pm.$ For $n \gg 0$ the map $T_n^\pm \to T_\infty^\pm$ is surjective, so
\[
|\mathrm{ker}(U_n^\pm \to U^\pm_\infty)| \leq |\mathrm{ker}(S_n^\infty \to S_\infty^\pm)|.
\]
By \cite[Lemma 9.1]{Kob}, the right-hand side is zero, so the maps $U_n^\pm \to U^\pm_\infty$ are also injective.  By Remark \ref{selfdualannihilator}, Flach's generalized Cassels-Tate pairing 
\[
U_n^\pm \times U_n^\pm \to \Q_p / \Z_p
\]
forces $|U_n^\pm|$ to be a square \cite[Corollary to Theorem 2]{flach}; combined with the injectivity of $U_n^\pm \to U^\pm_\infty$, it follows (as in the proof of \cite[Proposition 3.10]{Gr}) that $u = \mathrm{corank}_{\Z_p}\ U^\pm_\infty$ is even. Since $u = \lambda_E^\pm - \bar{\lambda}^\pm$, we have
\[
\lambda_E^\pm \equiv \bar{\lambda}^\pm \mod 2.
\]
Combining this with (\ref{lambdapar}) gives the desired result.
\end{proof}

Before continuing, we note an interesting corollary which we have not seen in the literature.

\begin{corollary}
For $E$ and $p$ as in the proposition, we have
\[
\lambda_E^+ \equiv \lambda_E^- \mod 2.
\]
\end{corollary}

\begin{proof}
This follows immediately from the preceding proposition.
\end{proof}

\begin{remark}
This corollary is not expected to hold in more general contexts, e.g. Iwasawa theory over anticyclotomic $\Zp$-extensions. See \cite[Remark 7.2]{IP} for a discussion of this.
\end{remark}

Let $r^{an}(E / \Q)$ denote the {\em analytic rank} of $E$ over $\Q$, which is the order of vanishing at $s=1$ of $L_{/ \Q}(E,s)$, its Hasse-Weil L-function over $\Q$. The following theorem is a special case of the parity conjecture for elliptic curves.

\begin{theorem}[parity conjecture]\label{parityconjecture}
Let $p \geq 3$ be an odd prime, and let $E$ be an elliptic curve defined over $\Q$ for which $p$ is supersingular. Then
\[
\mathrm{corank}_{\Z_p}\ \Selp(E / \Q) \equiv r^{an}(E / \Q) \mod 2. 
\]
\end{theorem}
\begin{proof}
This is a deep result of Nekovar \cite[Theorem A]{Nek}.
\end{proof}

We are now in a position to prove our main result. 

\begin{theorem}\label{main-theorem}
Let $E_1$ and $E_2$ be two elliptic curves defined over $\Q$ which have good supersingular reduction at a prime $p \geq 3$. Suppose $E_1[p]$  and $E_2[p]$ are isomorphic as $G_\Q$-modules. Assume that $\mu_{E_1}^\pm=0$ (hence also $\mu_{E_2}^\pm=0$). Define $\Sigma_0$ as in (\ref{sigma0}), and let $S_i$ be as in Definition \ref{sidef}. Then
\[
r^{an}(E_1 / \Q) + |S_1| \equiv r^{an}(E_2 / \Q) + |S_2| \mod 2,
\]
where $|S_i|$ denotes the cardinality of $S_i$.
\end{theorem}
\begin{proof}
The hypotheses of the theorem allow us to apply Proposition \ref{mainprop}, so that
\[
\lambda_{E_1}^+ + \delta(E_1,\Sigma_0) = \lambda_{E_2}^+ + \delta(E_2,\Sigma_0).
\]
By Theorem \ref{parityconjecture}, we have $\lambda_{E_i}^+ \equiv r^{an}(E_i / \Q) \mod 2$ for $i=1,2$. By Proposition \ref{deltacomp}, $\delta(E_i,\Sigma_0) \equiv |S_i| \mod 2$, and the theorem follows.\end{proof}

\section{Examples}\label{examplesection}

In these final two sections of the paper, we refer to a computer algebra system, or a {\em CAS}, to mean either Magma \cite{magma} or Sage \cite{sage}, whichever the reader prefers. 

\subsection{$p=5$}

Let 
\[
E_1 : y^2 + xy + y = x^3 - x - 1
\]
be the curve labeled $69a$ in the Cremona database, and let
\[
E_2 : y^2 + xy + y = x^3 + 130884x - 59725523
\]
be the curve labeled $897d$. Let $p=5$; one checks that both $E_1$ and $E_2$ are supersingular at $p$. Using Sturm's bound and a CAS, it can be verified that $E_1[p] \simeq E_2[p]$. By using a CAS, or by simply referring to Pollack's data \cite{PolData}, we find that $\mu_{E_1}^\pm = \mu_{E_2}^\pm = 0$, so Theorem \ref{main-theorem} applies to these two curves.

Using a CAS, we find that
\[
r(E_1 / \Q) = 0 \ \mathrm{and} \ r(E_2 / \Q) = 1.
\]
The prime-to-$p$ conductor of $E_i[p]$ is $69$ for both $i=1,2$, hence $\Sigma_0 = \{ 13 \}$. Let $\ell = 13$. Since $\ell \equiv 3 \mod p$ and $E_2$ has (split) multiplicative reduction at $\ell$, we see that $|S_2|=0$. On the other hand, since $\ell \nmid 69$, $E_1$ has good reduction at $\ell$. One finds that $a_\ell(E_1) = -6 \equiv 4 \mod p$, and so
\[
\ell + 1  \equiv a_\ell(E_1) \mod p \quad \text{with} \  \ell \not\equiv 1 \mod p,
\]
hence $|S_1|=1$. Thus
\[
r(E_1 / \Q) + |S_1| \equiv r(E_2 / \Q) + |S_2| \mod 2.
\]

\subsection{$p=3$} 

Consider the elliptic curve
\[
\mathcal{E}_D : y^2 = x^3 - Dx
\]
where $D$ is a nonnegative integer. In \cite{RS}, Rubin and Silverberg construct an explicit family of 
\[
\mathcal{E}(t) : y^2 = x^3 + D(27D^2t^4 - 18Dt^2 -1)x + 4D^2t(27D^2t^4 +1)
\]
such that $\mathcal{E}_D$ and $\mathcal{E}_t$ are congruent mod $p$ for every $t$. For the rest of the paper, let $E_1 = \mathcal{E}_1$, which Sage confirms has analytic rank $0$. One checks that $a_p(E_1)=0$, so $E_1$ is supersingular at $p$. If $3 \mid t$, then it is clear from the affine equations given above that $|\tilde{E}(\F_p)| = |\tilde{\mathcal{E}}(t)(\F_p)|$, so $a_p(\mathcal{E}_t)=0$ as well. 



Let $E_2 = \mathcal{E}(207)$. Sage hangs when asked to compute the analytic rank of $E_2 / \Q$, but it quickly returns an upper bound of $1$. 

For the pair $E_1, E_2$, we have $\Sigma_0 = \{ 37, 83, 4035637 \}$.

First let $\ell = 37 \equiv 1 \mod 3$.  Then $E_1$ has good reduction at $\ell$, and $a_\ell(E_1) = -2 \equiv 1 \mod p$. Since $\ell + 1 \not\equiv a_\ell(E_1) \mod 3$, we have $\ell \notin S_1$. On the other hand, $\ell \equiv 1 \mod 3$, but $E_2$ has nonsplit multiplicative reduction at $\ell$, so $\ell \notin S_2$. as well.

Next let $\ell = 83 \equiv 2 \mod 3$. Then $a_\ell(E_1)=0$, $E_1$ has good reduction at $\ell$, and $\ell + 1 \equiv a_\ell(E_1) \mod p$, so we have $\ell \in S_1$.  Since $E_2$ has split multiplicative reduction at $\ell$, we see $\ell \notin S_2$. 

Finally, let $\ell = 4035637 \equiv 1 \mod 3$. Again $E_1$ has good reduction, but $a_\ell(E_1) = 3598 \equiv 1 \mod 3$, so $\ell + 1 \not\equiv a_\ell(E_1) \mod p$, hence $\ell \notin S_1$. Since $E_2$ has nonsplit multiplicative reduction, we also have $\ell \notin S_2$. 

Thus

\[
r(E_1 / \Q) = 0 \ \text{and} \ |S_1|=1,
\]

\[
0 \leq r(E_2 / \Q) \leq 1 \ \text{and} \ |S_2|=0.
\]

Therefore, by Theorem \ref{main-theorem} we must have $r(E_2 / \Q) = 1$.

\vspace{5ex}

\section*{acknowledgements} It is a great pleasure to thank Keenan Kidwell for his careful reading and helpful comments on an earlier version of this paper. We are also very grateful to Antonio Lei for for suggesting many improvements over the original version of the paper.

\end{document}